\documentclass[11pt,a4paper,ralign]{amsart}
\usepackage{amsmath,amssymb,amsfonts,amscd}
\usepackage[mathscr]{euscript}
\usepackage[all]{xy}

\theoremstyle{plain}
\newtheorem{theorem}{Theorem}[section]
\newtheorem{lemma}[theorem]{Lemma}

\newtheorem{corollary}[theorem]{Corollary}
\theoremstyle{definition}

\newtheorem{remark}[theorem]{Remark}

\def\<#1>{\langle\, #1\,\rangle}

\newcommand{\dsp}{\displaystyle}

\newcommand{\mb}{\mathbf}

\newcommand{\R}{\mathbb{R}}

\newcommand{\C}{\mathbb{C}}

\newcommand{\N}{\mathbb{N}}

\newcommand{\G}{\mbox{\tiny $G$}}

\newcommand{\luc}{\mathit{LUC}(G)}

\usepackage{amsmath,amscd}

\font\seis=cmr6
\def\CB{\mathscr{CB}}
\def\CK{\mathscr{CK}}
\def\luc{{\seis{\mathscr{LUC}}}}

\def\wap{{\seis{\mathscr{WAP}}}}

\newcommand{\I}{\mathbb{I}}

\begin{document}
\title[Extreme Non-Arens Regularity of the Group Algebra ]%
      {Extreme Non-Arens Regularity of the Group Algebra }

\author[Filali and Galindo]{M. Filali \and  J. Galindo}
\thanks{ Research of  the second named  author  supported by the Spanish Ministry of Science (including FEDER funds), grant
MTM2011-23118 and  Universitat Jaume I, grant
P1$\cdot$1B2011-30.}

\address{\noindent Mahmoud Filali,
Department of Mathematical Sciences\\University of Oulu\\Oulu,
Finland. \hfill\break \noindent E-mail: {\tt mfilali@cc.oulu.fi}}
\address{\noindent Jorge Galindo, Instituto Universitario de Matem\'aticas y
Aplicaciones (IMAC)\\ Universidad Jaume I, E-12071, Cas\-tell\'on,
Spain. \hfill\break \noindent E-mail: {\tt jgalindo@mat.uji.es}}


\date{\today}

\begin{abstract}
Following Granirer, a Banach algebra $A$ is extremely non-Arens regular when the quotient space
${\dsp\frac{A^*}{\wap(A)}}$ contains a closed linear subspace which has $A^*$ as a continuous linear image.
We prove that the group algebra $L^1(G)$ of any infinite locally compact group is always  extremely non-Arens regular. When $G$ is not discrete, this result is deduced from the much stronger property that,
in fact, there is  a linear isometric copy of $L^\infty(G)$ in the quotient space ${\dsp \frac{L^\infty(G)}{\CB(G)}}$.
 \end{abstract}
\maketitle

\section{Introduction}
The second dual space $A^{**}$ of a Banach algebra $A$ can be made into a Banach algebra with two different products, each extending the original product of $A$. These products were introduced by Arens in 1951  and are called the first (or left) Arens product and the second (or right) Arens product,  see \cite{A1} and \cite{A2}.
We may describe the Arens products explicitly as follows (although we shall have no need of that):
If $(\mu_\alpha)$ and $(\nu_\beta)$ are nets in $A$ with $\lim_\alpha\mu_\alpha=\mu$ and $\lim_\beta\nu_\beta=\nu$,  then
\[\mu\circ\nu =\lim_\alpha\lim_\beta\mu_\alpha\nu_\beta\quad\text{and}\quad\mu\square\nu=\lim_\beta\lim_\alpha\mu_\alpha\nu_\beta,\]
where the limits are taken in the weak*-topology in $A^{**}$ and the order of the limits is crucial.
Note that $\mu\mapsto\mu\circ\nu$ is continuous in $\mu$ for each fixed $\nu\in A^{**}$
and is continuous in $\nu$ for each fixed $\mu\in A$.
In general, it is not continuous in $\nu$ when $\mu$ is not $A$.
The topological centre of $A^{**}$ is defined by
\begin{align*}
Z(A^{**})&=\{\mu\in A^{**}\;:\;\nu\mapsto \mu\circ\nu\quad\text{is continuous on}\;A^{**}\}\\&
=\{\mu\in A^{**}\;:\;\mu\circ\nu=\mu\square\nu\quad\text{for all}\;\nu\in A^{**}\}.\end{align*}
As already noted, $A$ is a subalgebra of $Z(A^{**}).$

The same observations and definitions may be given for the second product with the roles of the
variables reversed.

The algebra $A$ is said to be {\it Arens regular} if these two products coincide, which is the same as $Z(A^{**})=A^{**}$.
When $A$ is commutative, it is easy to check that $A$ is Arens regular if and only
if $A^{**}$ is commutative with respect to both products.

An interesting criterion for the regularity of $A$ was given by Pym \cite{P} in 1965 when he considered   the space $\wap(A)$ of weakly almost periodic functionals on $A.$ This is the space of all
functionals $a'\in A^*$ such that the set \[\{a'.a: a\in A \quad\text{and}\quad \|a\|\le1\}\] is relatively compact in the weak topology of $A^*$, where the action of $A$ on $A^*$ is given by
\[<a'.a,b>=<a',ab>,\quad a,\;b\in A,\; a'\in A^*.\] Pym went on showing that $A$ is Arens regular if and only if $\wap(A)=A^*$,  see also \cite{Yo}. For further details, see \cite{FiSi}.

It is known  that all $C^*$-algebras are Arens regular. This fact was first proved (implicitly)   by Sherman \cite{S}   and  Takeda \cite{T} when they proved that the second conjugate of a $C^\ast$-algebra is  a von Neumann algebra.  Some years later,  Civin and Yood   \cite[Theorem 7.1]{CY} reproduced  Takeda's  proof and  brought up explicitly the Arens regularity of $C^\ast$-algebras.
A different proof of this same fact can also be found in \cite[Theorem 38.19]{BD}.
For more details, see \cite{DH}, \cite{P}, \cite{D}, \cite{FiSi} and \cite{DL}.

The group algebra   $L^1(G)$ of an infinite locally compact group, however, is never Arens regular.
Arens himself showed that the semigroup algebra $\ell^1$ with convolution
is non-Arens regular. To prove this fact he produced two distinct invariant means $\mu$ and $\nu$ in ${\ell^1}^{**}.$ It is then trivial to see that $\nu\mu=\mu\ne\nu=\mu\nu$.
In \cite{Day}, Day used the same argument to show that $L^1(G)$ is non-Arens regular for many infinite discrete groups, including all Abelian ones.
This was followed by the seminal paper of Civin and Yood \cite{CY} where it was proved  that $L^1(G)$ is non-Arens
regular for any infinite locally compact Abelian group. Their method relied
again on Day's result when the group is discrete. When $G$ is not discrete, the Hahn Banach theorem provides a non-zero right annihilator $\mu$ of $L^1(G)^{**},$ i.e., $L^1(G)^{**}\mu=\{0\},$
and so $\mu \nu=\mu\ne 0=\nu\mu$ for every right identity $\nu$ in $L^1(G)^{**}.$

The general case was finally settled by Young in 1973.
Young relied on the criterion proved in \cite{P} and \cite{Yo} and produced a function in $L^\infty(G)$ which is not in $\wap(G).$

It is worthwhile to note that Young's approach     to non-Arens regularity of $L^1(G)$ was essentially different  from all previous ones. Depending on the approach, two extreme types of non-Arens regularity arise:  strong Arens irregularity and  extreme non-Arens regularity.

The Banach algebra $A$ is said to be {\it strongly Arens irregular}, according
to Dales and Lau \cite{DL}, when $Z(A^{**})=A$.
The group algebra is strongly Arens irregular for any infinite locally compact group.
This was proved first by Isik, Pym and \"Ulger \cite{IUP} when $G$ is compact, then by Grosser and Losert in \cite{GL} when $G$ is Abelian,
and finally in the general case by Lau and Losert in \cite{LL}.
A number of articles offering different approaches  and various properties related to the topological centres appeared subsequently, we cite for example \cite{DL}, \cite{N}, \cite{FS}
and most recently \cite{BIP} and \cite{FS2}.

The second type of irregularity, the one we focus on in this paper, was first studied  by Young in 1973 in \cite{Y} and actually stems from Pym's theorem on the equivalence between $\wap(A)=A^\ast$ and Arens regularity of $A$. This is the notion of extreme non-Arens regularity introduced by Granirer in \cite{G} that countered
   Pym's result some thirty years later.
Granirer said that the Banach algebra $A$ is {\it extremely non-Arens regular} if the quotient space ${\dsp \frac{A^*}{\wap(A)}}$ contains a closed linear subspace which has $A^*$ as a continuous linear image, and proved that the Fourier algebras
$A(\mathbb R)$ and $A(\mathbb T)$ are extremely non-Arens regular. Thus, the group algebras $L^1(\mathbb R)$ and $\ell^1(\mathbb Z)$ are extremely non-Arens regular.
In \cite{H},  Hu generalized Granirer's results and proved that the Fourier algebra $A(G)$ is extremely non-Arens regular whenever $w(G)\ge\kappa(G)$, where $\kappa(G)$ is the minimal number of compact sets required to cover $G$ and $w(G)$ is the minimal cardinality of an open base at the identity $e$ of $G$.

In \cite[Theorems 4.1 and 4.7]{FN}, Fong and Neufang considered infinite, locally compact, metrizable groups with $\kappa(G)\ge w(G)$ and proved that the group algebra $L^1(G)$ is extremely non-Arens regular when $G$ is
$\sigma-$compact and metrizable or it contains an open $\sigma-$compact, metrizable, subgroup $H$ which is either normal or has $|H|<|G|.$

In \cite{FV}, Filali and Vedenjuoksu proved, using slowly oscillating functions, that the semigroup algebra $\ell^1(S)$ is extremely non-Arens regular for any infinite weakly cancellative discrete semigroup. In particular, this property is held by the group algebra $L^1(G)$
for any infinite discrete group $G.$

In \cite{BF}, Bouziad and Filali, using a method  different from that of  \cite{FV}, proved that the quotient space ${\dsp \frac{\luc(G)}{\wap(G)}}$ contains a copy of $\ell^\infty(\kappa(G))$ for every locally compact group $G$. Thus,  the full dual of Hu's result was realized by showing that the group algebra $L^1(G)$ is extremely non-Arens regular whenever  $\kappa(G)\ge w(G)$.
The same conclusion was obtained again in \cite{FG1} for non-discrete  $G$,  using the size of the quotient space ${\dsp\frac{\CB(G)}{\luc(G)}}.$ Here, $\CB(G)$ is the space of all bounded, continuous, scalar-valued functions on $G$,
$\luc(G)$ is the space of bounded scalar-valued functions on $G$ which are uniformly continuous with respect to the right uniformity of $G$, and $\wap(G)$ is the space of bounded, continuous, scalar-valued functions on $G$ which are
weakly almost periodic, all with the supremum norm.

In this paper, we prove the  full theorem.
\medskip

\noindent
{\bf Theorem A.} {\it The group algebra $L^1(G)$ is extremely non-Arens regular for any infinite locally compact group.
}
\medskip

As already noted, this result was proved in  \cite{BF}     and \cite{FV} when $G$ is discrete. So we are really
concerned with the case when $G$ is not discrete. In this case, in fact, we shall prove the following stronger theorem.
\medskip

{\bf Theorem B.} {\it
There exists a linear isometric copy of $L^\infty(G)$ in the quotient space ${\dsp \frac{L^\infty(G)}{\CB(G)}} $ for any infinite, non-discrete, locally compact group $G.$}
\medskip

As for the Fourier algebra $A(G),$ the situation is even more complex.
When $G$ is an amenable locally compact group, the Fourier algebra $A(G)$ is Arens regular if and only if $G$ is finite. This was proved in 1989 by Lau and Wong, see \cite{LW}.
Further results on Arens regularity of $A(G)$ were published two years later by Forrest in \cite{F}.
The question of whether $A(G)$ is not Arens regular
seems to be still not completely settled for discrete non-amenable groups.
As already mentioned, Hu proved in \cite{H} that $A(G)$ is extremely non-Arens regular whenever $w(G)\ge \kappa(G)$. Our Theorem A enables us to omit this condition
when $G$ is Abelian.
\medskip

It may worthwhile to note that
extreme non-Arens regularity in the sense of Granirer does not imply strong Arens irregularity in the sense of \cite{DL} since $A(SO(3))$ is extremely non-Arens regular by \cite{H}, but it is not strongly Arens irregular as recently proved by Losert \cite{L}. Neither does strong Arens irregularity imply extreme non-Arens regularity see, \cite{HN2}.

\subsection{The spaces}
We recall that
 $L^1(G)$ is  the Banach space made of
all    equivalence classes of scalar-valued functions which are integrable with respect to the Haar measure $\lambda_{\G}$, where, as usual, two functions are  equivalent if and only if they differ only on a set of Haar measure zero.

   By $L^\infty (G)$ we understand the Banach dual of $L^1(G)$.
Our incarnation  of $L^\infty(G)$,  valid for a general locally compact group, requires the concept of locally null set.  A subset $A\subset G$ is \emph{locally null} if $A\cap K$ is of Haar measure zero  for every compact set $K\subset G$ and two functions $f,g\colon G\to \C$ are   equal \emph{locally almost everywhere} (l.a.e.),  if there is a locally null subset $A$ such that $f(x)=g(x)$ for all $x\notin A$.
  Following  \cite[Section 12]{hewiross1} or \cite[Page 46]{Fo} (and departing  from the most usual definition), we identify $L^\infty(G)$ with
  the vector space of all equivalence classes of essentially bounded and locally   measurable functions,  two functions being equivalent if and only if they  are equal locally almost everywhere.
   If $G$ is $\sigma$-compact (so that Haar measure is $\sigma$-finite), a subset $A\subset G$ is null if and only if it is locally  null, hence  this construction of $L^\infty(G)$  coincides with the standard one as $L^\infty(G,\lambda_{\G})$.

We may also recall that a function $f\in \CB(G)$ is {\it weakly almost periodic} when the set of its left
(or equivalently, right) translates $\{f_s:s\in G\}$  is relatively weakly compact in $\CB(G);$ and that Grothendieck's famous iterated limit criterion
shows that a function $f\in \CB(G)$ is in $\wap(G)$ if and only if, for any sequences $(x_n)$ and
$(y_m)$ in $G$, \[\lim_{n\to\infty}\lim_{m\to\infty} f(x_ny_m)=
\lim_{m\to\infty}\lim_{n\to\infty} f(x_ny_m)\]
whenever these iterated limits exist; see for example
\cite[Appendix A]{BJM}.

The reason for the space $\luc(G)$ to be denoted as such in this paper and elsewhere in the literature
is that it may also be given as
\[\luc(G)=\{f\in \CB(G)\;:\;s\mapsto f_s\quad\text{is norm continuous}\}.\]
Functions with this  property are usually called  {\it left norm continuous} and can be defined in any semitopological semigroup where a natural uniformity may not be available, see  \cite[Section 4.4]{BJM}.

We recall that $\wap(L^1(G))=\wap(G)$ (see \cite{U}) and that
\[\wap(G)\subseteq \luc(G)\subseteq \CB(G),\] see for example \cite[Theorem 4.4.10]{BJM} for the first inclusion.
This means that if there is a linear isometric copy of $L^\infty(G)$ in the quotient space ${\dsp\frac{L^\infty(G)}{\CB(G)}}$,
the same copy is also  in the quotient  ${\dsp\frac{L^\infty(G)}{\wap(G)}}$, and so the group algebra $L^1(G)$
is extremely non-Arens regular. We shall use this observation without any reference.
\medskip

\subsection{Outline of our approach}
In each of the papers cited above dealing with the extreme non-Arens regularity of $L^1(G)$, the proof relied on finding a copy of $\ell_\infty(d)$ in the quotient space ${\dsp\frac{L^\infty(G)}{\wap(G)}}$,
where $d=  \max\{\kappa(G),w(G)\}$ is the density character of $L^1(G)$, then embedding  $L^\infty(G)$  in $\ell_\infty(d)$. This method was efficacious when $\kappa(G)\ge w(G)$ or when $G$ is metrizable.
However, it is very likely to fail for non-metrizable compact groups. In fact, Rosenthal proved in \cite[Proposition 4.7, Theorem 4.8]{rose70}, that
regardless of the size of the base at the identity of $G$, when $G$ is compact, $\ell_\infty(\kappa)$ does not embed in $L^\infty(G)$ (and so not in the quotient space ${\dsp\frac{L^\infty(G)}{\wap(G)}}$ either)  when $\kappa>\omega$.
So we need a different and a more elaborate method to reach our aim. Our strategy shall rely on the following three keys:
\begin{itemize}
\item We first deal with $\sigma$-compact groups of the form  $G=M\times H,$
 with $M$ a  \emph{ non-discrete, metrizable and $\sigma$-compact} group. For such a group $G$, we show in Theorem  \ref{main:constrvv2}
 that the  quotient ${\dsp\frac{L^\infty(G)}{\CB(G)}}$  contains a copy of $\ell^\infty(L^\infty(H))$, a result inspired by  \cite[Theorem 2.11]{FG1} whose ideas date back at least to \cite{C1}.

This theorem is the first tool and is already used in {\it Corollary \ref{prodenar}} to deduce that there is a linear isometric copy of $L^\infty(G)$ in the quotient space ${\dsp\frac{L^\infty(G)}{\CB(G)}}$ whenever $G$  contains an open subgroup which is Haar homeomorphic to a locally compact group of the form given above.
 So $L^1(G)$ is extremely non-Arens regular for such groups.

\item The second key is Lemma \ref{lem:cond}, which comes in  Section \ref{7Isom}, a technical section.
The lemma isolates a set of  conditions which go in the spirit of the previous and forthcoming sections and
which provide seven linear isometries required for the final task.
This lemma  reveals  how to glue together these seven linear isometries giving a linear isometric copy of $L^\infty(G)$ in the quotient space ${\dsp\frac{L^\infty(G)}{\CB(G)}}$, and so the extreme non-Arens regularity of the group algebra.

\item The work by Grekas and Mercourakis \cite{GM} is our third  main tool and is presented in section \ref{7Isom}.
It proves that every compact group $K$ can be sandwiched between two products of metrizable groups.
This, together with Lemma \ref{lem:cond} and  our first key, will lead to  {\it Corollary \ref{cor:compenar}} to the effect that the quotient space ${\dsp\frac{L^\infty(K)}{\CB(K)}}$ contains a linear isometric copy of $L^\infty(K))$ for any infinite compact group.
Thus, $L^1(K)$ is extremely non-Arens regular for any infinite compact group $K$.
\end{itemize}

Once these three keys are available, the final step will be  based on theorems  due to Davis \cite{davis55} and Yamabe \cite{Ya}. The problem of finding a linear isometric copy of $L^\infty(G)$ in ${\dsp\frac{L^\infty(G)}{\CB(G)}}$ as well as of extreme non-Arens regularity of the group algebra of any locally compact group is reduced to that of $L^1(\R^n\times K)$, where $K$ is a compact group.
 {\it Corollary \ref{prodenar}} and {\it Corollary \ref{cor:compenar}} lead then immediately to the proof of Theorem B.

Since we know from \cite{BF} or  \cite{FV} that the group algebra is extremely non-Arens regular for any infinite discrete group, Theorem A is then an immediate consequence.

\section{Notation and terminology}
 Being concerned as we are with isometries into quotients of $L^\infty(G)$-spaces we deal with   locally compact groups and their Haar measures and with Banach spaces  and isometries between them.
 We summarize here  our  terminology and some basic results to be used throughout the paper.
\subsection{Topological groups and Haar measure}
The density character of a topological space $X$ is the least cardinality of a dense subset of $X$. We denote it by $d(X).$

The local weight $w(G)$ of locally compact group $G$ is the least cardinality of an open base at the identity of $G$.
The compact covering $\kappa(X)$ of a topological space $X$ is the least cardinality of a compact covering of $X$.

If $G$ is a locally compact group $\lambda_{\G}$ will denote the left Haar measure of $G$.
The characteristic function of  a set $V$ will be denoted as $\chi_V$.
\subsection{Banach spaces and isometries}
If $E$ is a Banach space  and $I$ is  a nonempty index  set, $\ell^\infty(I ,E)$ will denote as usual the linear space of all families $\mathbf{x}=(x_i)_{i\in I}$ with $x_i \in E$ and with $\sup\left\{\|x_i\|_{E}\colon i\in I\right\}<\infty$. Equipped with the norm
$\|\mathbf{x}\|=\sup\left\{\|x_i\|_E\colon i\in I\right\}$, $\ell^\infty(I,E)$ turns into a Banach space. The particular case $\ell^\infty(\omega,E)$ will be be denoted simply as $\ell^\infty(E)$. Of course, when $|I|=1$, $\ell^\infty(I,E)$ is just the Banach space $E$. Cardinal numbers will be identified with their initial ordinals, so that the above definition makes sense for $\ell(\alpha,E)$ when $\alpha$ is a cardinal number.

By an isometry between  two Banach spaces $E_1$ and $E_2$ we understand a map $T_1\colon E_1\to E_2$ with $\|T(v)\|_{E_2}=\|v\|_{E_1}$ for all $v\in E_1$. Note that we do not assume that isometries are onto. The following straightforward lemma will help to smoothen some of the proofs.

\begin{lemma}
  \label{lem:quotiso}
  Let $E_1$ and $E_2$ be  Banach spaces, let $F_1$ and $F_2$ be  closed subspaces of $E_1$ and $E_2$, respectively, and let $\alpha$ be a cardinal number. Suppose  $T\colon E_1\to E_2$ is a linear  isometry that satisfies the following properties:
  \begin{align}
    \label{*}\tag{$\ast$}
    &T(F_1)\subset F_2\; \mbox{ and}\\    \label{**}\tag{$\ast\ast$}
&\mbox{\it \small for each  $x\in F_2$,  there is $y_x\in F_1$ with
$\|T(v)-x\|_{E_2}\geq \|v-y_x\|_{E_1}$, for every $v\in E_1$}.\end{align}
 Then, $T$ induces a linear isometry $\begin{CD}
  {\displaystyle
  \widetilde{T}\colon {\dsp\frac{\ell^\infty(\alpha,E_1)}{\ell^\infty(\alpha,F_1)}}}@>>>
{\dsp\frac{\ell^\infty(\alpha,E_2)}{\ell^\infty(\alpha,F_2)}}.
\end{CD}$
\end{lemma}

\begin{proof}
  We first prove  the case $\alpha=1$.

Conditions \eqref{*} and \eqref{**} imply  that $T(F_1)=F_2\cap T(E_1)$. Therefore,  if
  $q_i \colon E_i\to {\dsp\frac{E_i}{F_i}}$, $i=1,2$ denote the respective quotient mappings, we can define a 1-1 map $\widetilde{T}$ in such a way that the following diagram commutes:
\[\begin{CD}
E_1@>T>>E_2\\
@VVq_1V@VVq_2V\\ \dsp  {\dsp\frac{E_1}{F_1}}@>\widetilde{T}>>{\dsp \frac{E_2}{F_2}}
\end{CD}\]
If $v\in E_1$ and $x\in F_2$, condition \eqref{**} implies that $\|T(v)-x\|_{E_2}\geq \|q_1(v)\|_{\frac{E_1}{F_1}}$, therefore \[\|\widetilde{T}(q_1(v))\|_{\frac{E_2}{F_2}}=\|q_2(T(v))\|_{\frac{E_2}{F_2}}\geq \|q_1(v)\|_{\frac{E_1}{F_1}}.\]
The other  inequality is even simpler:
\begin{align*}
  \|\widetilde{T}(q_1(v))\|_{E_2/F_2}&=\inf\left\{\|T(v)-x\|_{E_2}\colon x\in F_2\right\}\\
  &\leq \inf\left\{\|T(v)-T(y)\|_{E_2}\colon y\in F_1\right\}\\
  &= \inf\left\{\|v-y\|_{E_1}\colon y\in F_1\right\}\\
  &=\|q_1(v)\|_{E_1/F_1}.
\end{align*}
Thus, $\widetilde{T}$ is an isometry, it is obviously linear.

To deduce the general case from the case $\alpha=1$, it is enough to observe that ${\displaystyle {\dsp\frac{\ell_\infty(\alpha, E_i)}{\ell_\infty(\alpha, F_i)}}}$ is naturally isometric to
${\displaystyle \ell_\infty\left(\alpha,{\dsp\frac{E_i}{F_i}}\right)}$,  a fact that can be routinely proved.
\end{proof}
Note that conditions \eqref{*} and \eqref{**} are automatically satisfied when $T(F_1)=F_2$.

\section{${\dsp \frac{L^\infty(G\times H)}{\CB(G\times H)}}$ contains $\ell^\infty(L^\infty(H))$}

We start with our first key theorem. The metrizability and the non-discreteness of the factor group $M$ are used to get a copy of $\ell^\infty(L^\infty(H))$ in the quotient space ${\dsp\frac{L^\infty(M\times H)}{\CB(M\times H)}}$.
This shall be applied in each of the theorems leading to the extreme non-Arens regularity of the group algebra.
This section produces isometry $\mathbf\Psi_5.$

\begin{theorem}\label{main:constrvv2}
 Let $H$ be any $\sigma$-compact locally compact group and let $M$ be a  \emph{non-discrete metrizable } $\sigma$-compact group. Then there exists a linear  map
 \[\Psi\colon \ell^\infty(L^\infty(H))\to L^\infty(M\times H)\]
 which satisfies the following properties:
 \begin{enumerate}
   \item $\|\Psi(\mathbf{\xi})\|=\|\mathbf{\xi}\|$,
       \item  $\|\Psi(\mathbf{\xi})-\phi\|\geq \|\mathbf{\xi}\|$
 \end{enumerate}
 for every $\mathbf{\xi}\in \ell^\infty(L^\infty(H))$ and $\phi \in \CB(M\times H)$.
%
\end{theorem}

\begin{proof}
  Let $\{U_n\}_{n\in \N}$ denote a  basis of neighborhoods of the identity in $M$  with $U_{n+1}^2\subset U_n$ for every $n\in\N$.
 Put, for each $n\in\N$, $V_n=U_n\setminus \overline{U_{n+1}}$, and consider the disjoint family of open sets in $M$ given by $\{V_n:n\in\N\}$.
Next, partition $\N$ into $\omega$-many subsets $I_m$ each of cardinality $\omega$. Enumerate each $I_m$ as  $I_m=\{m_n\colon n\in \N\}$, with $m_n<m_{n+1}$.

For each $m\in \N$,  define $f_m\colon M\times H\to \C$ as
\[f_m(s,t)=\begin{cases}
   1, \quad \mbox{ if   } s\in V_{m_{2n}},\\
   -1, \quad \mbox{if   } s\in V_{m_{2n+1}},\\
   0,\quad \mbox{ otherwise}.
 \end{cases}\]

 Then
  $f_m\in L^\infty(M\times H)$, note that $f_m$   is simply
  \[{\displaystyle  f_m =\chi_{_{(\bigcup_n V_{m_{2n}})\times H}}-\chi_{_{(\bigcup_n V_{m_{2n+1}})\times H}}}.\]

  Define now for each
$\mathbf{\xi}=(\xi_m)_{m<\omega}\in\ell^\infty(L^\infty(H))$, the
function $f_{\mathbf{\xi}}\in L^\infty(M\times H)$  by
\[f_{\mathbf{\xi}}(s,t)=\sum_{m<\omega} \xi_m(t) f_m(s,t).
\]
Observe that $f_\xi$ is well-defined because the countable union of  null sets is null.

Then, consider the map $\Psi\colon \ell^\infty(L^\infty(H))\to
L^\infty(M\times H)$ given by \[ \Psi\left(\mathbf{\xi}\right)=
f_{\mathbf{\xi}}.\]
 It is easily observed   that $\Psi$ is linear and
 $\|\Psi(\mathbf{\xi})\|=
\|f_{\mathbf{\xi}}\|\leq \|\mathbf{\xi}\|$
for every $\mathbf{\xi}=(\xi_n)_{n<\omega}\in
\ell^\infty(L^\infty(H))$.

Moreover, since for each $s\in V_{m_{2n}}$, $f_{\mathbf{\xi}}(s,t)=\xi_m(t),$ we deduce that \[\|\Psi(\mathbf{\xi})\|=
\|f_{\mathbf{\xi}}\|=\|\mathbf{\xi}\|.\]

We now prove Statement (ii).
Since $\Psi$ is linear, we may assume without loss of generality that $\|\mb{\xi}\|=1$, and so we must show that \[\|f_\mb{\xi}-\phi\|\geq 1\quad\text{ for every}\quad \phi \in \CB(M\times H).\]
Suppose, otherwise, that for some $\varepsilon>0$ and some $\phi  \in \CB(M\times H)$, we have
\[\|f_\mb{\xi}-\phi\|< 1-\varepsilon.\]
Since $\|\mb{\xi}\|=\sup\{\|\xi_m\|:m\in\N\}=1,$ we may pick and fix $m\in\N$ such that \begin{equation}\label{norm1}\|\xi_m\|>1-\varepsilon/2.\end{equation}

 By Fubini's theorem,
 we can find a   null set $A\subset M$ such that
 $s\in M\setminus A$ implies that   the set \[C_s=\{t\in H\colon |    f_\mb{\xi}(s,t)-\phi(s,t)|\geq 1-\varepsilon\}\]
is null. The set $A$ can actually be described as \[A=\{s\in G:\lambda_H\left(C_s\right)>0\}.\]

 Let now $n\in\N$ be picked arbitrarily.

Being null, $A$ cannot contain either of the open sets $V_{m_{2n}}$ or $V_{m_{2n+1}}$.
We then choose $s_{m_{2n}}\in V_{m_{2n}}\setminus A$ and $s_{m_{2n+1}}\in V_{m_{2n+1}}\setminus  B$. Since $f_m$ takes the value 1 on $V_{m_{2n}}\times H$ and the value -1 on $V_{m_{2n+1}}\times H$, we see   that, for every $t\notin  C_{s_{m_{2n}}}\cup C_{s_{m_{2n+1}}}$:
\begin{align}
 \label{null1} \,\left|    \xi_m(t)-\phi(s_{m_{2n}},t)\right|\,&<1-\varepsilon, \mbox{ and  }
\\\label{null2}
\left|    \xi_m(t)+\phi(s_{m_{2n+1}},t)\right|&<1-\varepsilon.
\end{align}

On the other hand, our assumption \eqref{norm1} gives  a  null set $C_0\subset H$ such that \begin{equation}\label{null3}|\xi_m(t)|\geq 1-\displaystyle \frac{\varepsilon}{2}\; \mbox{ for every } t\in H,\;\; t\notin C_0.\end{equation}
Choose finally $t\notin C_0 \bigcup \left(\bigcup_{k} C_{s_{m_{k}}}\right)$ and suppose that $\xi_m(t)>0$. Then,  by  \eqref{null1}, \eqref{null2} and \eqref{null3}, we have for every $n\in\N$,
\[\phi(s_{m_{2n}},t)>\varepsilon/2\quad\text{ while}\quad
\phi(s_{m_{2n+1}},t)<-\varepsilon/2.\] If we observe that, by construction, $(s_{m_n})$ converges to $e$, we find that these last inequalities go against the continuity of $\phi$ at $(e,t)$. If $\xi_m(t)<0$, a similar argument leads to the same contradiction.
\end{proof}

We come to the fifth linear isometry needed for our diagram.

\begin{corollary}\label{cor:mainvv}
  Let $\alpha$ be an infinite cardinal, and let $M$ and $H$ be locally compact $\sigma$-compact groups with $M$ non-discrete and metrizable. Then  there exists a linear isometry \[\mathbf\Psi_5: \ell^\infty(\alpha,L^\infty(H))\to{\dsp\frac{\ell^\infty(\alpha,\;L^\infty(M\times H))}{\ell^\infty(\alpha,\;\CB(M\times H))}}.\]
\end{corollary}

\begin{proof}
   Lemma \ref{lem:quotiso} applied to  the isometry $\Psi$ constructed in Theorem \ref{main:constrvv2} and the subspaces  $F_1=\{0\} $ and $F_2=\CB(M\times H)$ gives a linear isometry
    \[
 \xymatrix{ \dsp \widetilde{\Psi}\colon \ell^\infty\bigl(\alpha,\ell^\infty\left(L^\infty(H)\right) \bigr)\ar[r]&\dsp
  \frac{\ell^\infty\left(\alpha,L^\infty(M\times H)\right)}{\ell^\infty\left(\alpha,\CB(M\times H)\right)}}.\]

Using a bijection between $\alpha\times \omega$ and $\alpha$ we may identify  the Banach spaces $\ell^\infty(\alpha,L^\infty(H))$ and  $\ell^\infty(\alpha\times\omega,L^\infty(H))$. This latter space is identified with $\ell^\infty(\alpha,\ell^\infty(L^\infty(H)))$ through the linear isometry
\[((\xi^{(n)}_\eta)_{n<\omega})_{\eta<\alpha}\mapsto (\xi_{(\eta,n)})_{(\eta,n)\in \alpha\times\omega},\]
where, for each $\eta<\alpha$, $(\xi^{(n)}_\eta)_{n<\omega}\in \ell^\infty(L^\infty(H))$ and $\xi_{(\eta,n)}=\xi^{(n)}_\eta$ for each $(\eta,n)\in \alpha\times \omega$.

In this way the isometry $\widetilde{\Psi}$ turns into the desired isometry $\mathbf{\Psi_5}$.
\end{proof}

\section{$L^\infty(M\times H)$ is contained in $\ell^\infty(L^\infty(H))$}
This section will produce the first, the third and the seventh of the planned linear isometries.

We first recall that, for any given locally compact group,   $d(L^p(G))\le \max\{\kappa(G),w(G)\}$, if  $1\le p<\infty$, see \cite[Lemma 7.3]{HN1}.
In particular, $d(L^1(G))=\omega$ when $G$ is an infinite, $\sigma$-compact, metrizable locally compact group.

Statement (ii) of next lemma already gives the third linear isometry $\mathbf\Psi_3.$

\begin{lemma}\label{Lintoell} Let $M$ and $H$ be
locally compact groups and put $d=d(L^1(M))$. If $\alpha$ is
be an infinite cardinal, then the following statements hold.
 \begin{enumerate}
\item There is a linear isometry of $L^\infty(M\times H)$ into $\ell^\infty(d, L^\infty(H))$.
\item If $d\le\alpha$, then there is a linear isometry \[\mathbf\Psi_3:\ell^\infty(\alpha, L^\infty(M\times H))\to \ell^\infty(\alpha, L^\infty(H)).\]
\end{enumerate}
\end{lemma}
\begin{proof} Let $\{u_i:i<d\}$ be a norm dense subset in the unit ball of $L^1(M)$.
Let $\lambda_M$ and $\lambda_H$ be fixed left Haar measures on $M$ and $H,$ respectively.
 We define a map: \begin{align*}\Psi: L^\infty(M\times H)&\to \ell^\infty(d,L^\infty(H))\\
f&\mapsto (\Psi(f)_i)_{i<d},\end{align*}
where $\Psi(f)_i$ is given for every $v\in L^1(H)$ by
\begin{align*} <\Psi(f)_i,v>=\int_M \left(\int_H f(s,t)u_i(s)v(t)d\lambda_H(t)\right)d\lambda_M(s).\end{align*}
We claim than $\Psi$ is a linear isometry of $L^\infty(M\times H)$ into $\ell^\infty(d, L^\infty(H))$.

 $\Psi(f)_i$   defines obviously a linear functional on $L^1(H)$ for each $i<d$. In addition,
\begin{equation}\label{cont}\begin{split}
\left|<\Psi(f)_i,v>\right|  &\leq \int_M \left(\int_H |f(s,t)|\cdot |u_i(s)|\cdot |v(t)|d\lambda_H(t)\right)d\lambda_M(s)\\
&\leq \left\|f\right\|_{_{L^\infty(M\times H)}}\cdot \left(\int_M|u_i(s)|d\lambda_M(s) \right)\cdot \left(\int_H |v(t)|d\lambda_H(t)\right)\\&\leq \|f\|_{_{L^\infty(M\times H)}} \|v\|_{L^1(H)}.\end{split}
\end{equation}
Being a continuous linear functional on $L^1(H)$, $\Psi(f)_i$ can  indeed be identified with an element of $L^\infty(H)$.

Inequality   \eqref{cont} actually gives:
\[\|\Psi(f)\|_{\ell^\infty(d,L^\infty(H))}\leq \left\|f\right\|_{L^\infty(M\times H)}.\]

We now check the converse inequality. For every $i<d$ and every function $v$ in the unit ball of $L^1(H)$, we have
\begin{align*}\|\Psi(f)\|_{\ell^\infty(d,L^\infty(H))}&=
\sup_{i<d}\|\Psi(f)_i\|_{L^\infty(H)}
\\&=
\sup_{i<d}\sup\left\{|<\Psi(f)_i,w>|:\|w\|_{L^1(H)}\le 1\right\}
\\&
\ge
|<\Psi(f)_i,v>|
\\&=
\left|
\int_{M} \left(\int_H f(s,t)u_i(s)v(t)d\lambda_H(t)\right)d\lambda_M(s)\right|.
\end{align*}
Accordingly, \[\|\Psi(f)\|_{\ell^\infty(d,L^\infty(H))}\ge \left|\int_M \left(\int_H f(s,t) u(s) v(t)d\lambda_H(t)\right)d\lambda_M(s)\right|\]
for every $u$ and $v$ in the unit balls of $L^1(M)$ and $L^1(H)$, respectively.

Now, the set of maps \[\{(s,t)\mapsto u(s)v(t)\colon u\in L^1(M) \mbox{ and } v\in L^1(H) \mbox{ with } \|u\|\leq 1,\; \|v\|\leq 1\}\] is dense in the unit ball of $L^1(M\times H)$ (it contains the characteristic function of every measurable rectangle of   integrable sides). Therefore,
\[\|\Psi(f)\|_{\ell^\infty(d,L^\infty(H))}\ge
\left|\int_M \left(\int_H f(s,t) h(s,t) d\lambda_H(t)\right)d\lambda_M(s)\right|\] for every $h$ in the unit ball of $L^1(M\times H)$. In other words,
\[\|\Psi(f)\|_{\ell^\infty(d,L^\infty(H))}\ge \|f\|_{L^\infty(M\times H)},\] as required for the first statement.

The second statement is a direct consequence of the first one.
Applying Lema \ref{lem:quotiso} to the isometry $\Psi$  and the subspaces $F_1=F_2=\{0\}$ and using  $d\le\alpha$ to identify, as in Corollary \ref{cor:mainvv}, $\ell^\infty(\alpha,\ell^\infty(d,L^\infty(H)))$ with $\ell^\infty(\alpha,L^\infty(H))$,  we get the linear isometry:
\[\mathbf\Psi_3:\ell^\infty(\alpha, L^\infty(M\times H))\to  \ell^\infty(\alpha, L^\infty(H)).\]
\end{proof}

We can already give the first consequence on the extreme non-Arens regularity of the group algebra
for some locally compact groups.


Following \cite{GM}, we say that two locally compact groups  groups $G_1$ and $G_2$ with Haar measures  $\lambda_1$ and $\lambda_2,$
respectively, are {\it Haar homeomorphic} when there exists a homeomorphism $\psi:G_1\to G_2$ such that  $A\subset G_1$ is $\lambda_1$-measurable if and only if $\psi(A)$ is $\lambda_2$-measurable and $\lambda_1(A)=\lambda_2(\psi(A))$. Note that in this situation, the map $f\mapsto f\circ \psi$ establishes a linear isometry between the group algebras $L^\infty(G_1)$
 and $L^\infty(G_2)$  as well as between the  $C^*$-algebras $\CB(G_1)$ and $\CB(G_2)$, and so these spaces  may be identified.

 If the groups are $\sigma$-compact, this definition is  equivalent to the condition \[\int_{G_2}f(y)d\lambda_2=\int_{G_1}f(\psi(x))d\lambda_1\]
 for every $f\in L^1(G_2)$.

 For the proof of the following corollary, we shall need to construct two of the seven linear isometries, the first
 and the seventh.
\begin{corollary} \label{prodenar}
Let $G$ be a locally compact group with an open $\sigma$-compact subgroup $G_0$ which is Haar homeomorphic to a locally compact group of the form $M\times H,$ where $H$ is a locally compact group and $M$ is a non-discrete, metrizable  group.
 Then there exists a linear isometry from  $L^\infty(G)$ into $  {\dsp \frac{ L^\infty(G)}{\CB(G)}}$.
 In particular, the group algebra
$L^1(G)$ is extremely non-Arens regular.
\end{corollary}

\begin{proof}
If $G$ is Haar homeomorphic to $G_0=M\times H,$ then Lemma \ref{Lintoell} and Corollary \ref{cor:mainvv} with  $\alpha=1$ give immediately
the chain of linear isometries
\[L^\infty(G)=L^\infty(M\times H)
\longrightarrow
\ell^\infty \left(L^\infty(H)\right)
\longrightarrow
\frac{\dsp L^\infty(M\times H)}{\CB(M\times H)}
=
\frac{\dsp  L^\infty(G)}{\CB(G)},
\]
leading to the claim.

If $G_0$ is a proper subgroup of $G$,
let $\alpha=|G:G_0|$ and  $\{x_\eta\colon  \eta<\alpha\}$ be a system of representatives of the right cosets of $G_0$ in $G$.

Now let $(\phi_\eta)_{\eta<\alpha}\in \ell^\infty\left(\alpha, L^\infty(G_0)\right),$
and for each $\eta<   \alpha,$ let $f_\eta $ be the function defined on $x_\eta G_0$ by $f_\eta(g)=\phi_\eta(x_\eta^{-1}g)$.  Define then $f:G\to \C$ by $f=f_\eta$ on $x_\eta G_0$, i.e., if the extensions to $G$ by zero of the functions $\phi_\eta$ and $f_\eta$ are denoted by the same letters,  then we may write
\[f=\sum_{\eta<\alpha} f_\eta \chi_{x_\eta G_0}.\]
  Since $A\subset G$ is locally null if and only if $A\cap x_\eta G_0$ is locally null for every $\eta<\alpha$ (see  \cite[Page 46]{Fo}),  the correspondence
\begin{equation} \begin{split} T:\;  \ell^\infty\left(\alpha, L^\infty(G_0)\right)&\longrightarrow  L^\infty(G)\\\label{isom0}
 (\phi_{\eta})_{\eta<\alpha}&\longmapsto f\end{split}\end{equation}
is a well-defined surjective linear isometry. We let  $\mathbf \Psi_1=T^{-1}.$

Furthermore, since
$\{ x_\eta G_0:\eta<\alpha\}$ is a  disjoint cover of $G$ made of closed and open sets, we see that the restriction of $T$ to $\ell^\infty(\alpha,\CB(G_0))$  is also a linear isometry onto $\CB(G)$, giving
raise, via Lemma \ref{lem:quotiso}, to an isometric isomorphism
\begin{equation}\label{isom}  \mathbf\Psi_7:\;\frac{\dsp \ell^\infty\left(\alpha, L^\infty(G_0)\right)}{\dsp \ell^\infty\left(\alpha,\CB(G_0)\right)}\longrightarrow  {\displaystyle \frac{ L^\infty(G)}{\dsp\CB(G)} }.\end{equation}

With Lemma \ref{Lintoell} (ii) and  Corollary \ref{cor:mainvv},  and recalling that $G_0$ is Haar homeomorphic to $M\times H$, we obtain  the chain of linear isometries
\[\ell^\infty\left(\alpha,L^\infty(G_0)\right)
\longrightarrow
\ell^\infty \left(\alpha, L^\infty(H)\right)
\longrightarrow
\frac{\dsp  \ell^\infty\left(\alpha,L^\infty(G_0)\right)}{\ell^\infty\left(\alpha,\CB(G_0)\right)}.\]

In view of the identifications made by the isometries $\mathbf{\Psi_1}$ and $\mathbf{\Psi_7}$ deduced from \eqref{isom0} and \eqref{isom},  this implies the existence of a linear isometry from  $L^\infty(G)$ into ${\dsp\frac{ L^\infty(G)}{\CB(G)} }$.
\end{proof}


\section{Seven isometries and a lemma}\label{7Isom}

So  far four out of the seven required isometries have been produced. This section points where to look for  the rest, namely, $\mathbf\Psi_2$, $\mathbf\Psi_4$ and $\mathbf\Psi_6,$ and so we reach in Lemma \ref{lem:cond} the main key lemma in the paper.

The previous sections, ending with Corollary \ref{prodenar}, show  that Theorems A and B hold in particular for
every locally compact group which contains a compact open  $K$ of the form $M\times H$ with $M$ metrizable, non-discrete and $\sigma$-compact.
In  this section, we prove that if $K$ can be sandwiched between two such groups
$M_1\times H_1$ and $M_2\times H_2$ via two continuous, open surjective homomorphisms $\varphi_1$ and $\varphi_2$
with $L^1(H_1)$ and $L^1(H_2)$ being isometrically isomorphic (showing the need of $\mathbf\Psi_4$), then  Theorems A and B hold as well.
 With the help of   Lemma \ref{chou}, the homomorphisms $\varphi_1$ and $\varphi_2$ will induce the isometries $\mathbf\Psi_6$ and $\mathbf\Psi_2,$  respectively.

The second part of our preparation starts with  Theorem \ref{Psi4}, and  identifies the $L^1$-spaces of products of uncountable,
separable, atomless, metric spaces. This theorem will show in the final section the availability of isometry
$\mathbf\Psi_4$.

  \medskip
We begin with a lemma that isolates a usual tool related to  the Weil formula, essential in the construction of Haar measure on quotients (see, for instance Theorem 2.56 of \cite{Fo} or Remark 3.4.1 of \cite{reitsteg}).
\begin{lemma}\label{weil}
Let $\varphi \colon K\to H$ be a continuous and open homomorphism of locally compact groups with compact kernel $N$. Then every $g\in L^\infty(K)$ induces a function $g^N\in L^\infty(H)$ with the following properties:
\begin{enumerate}
  \item $\|g^N\|\leq \|g\|$.
  \item $(f\circ \varphi)^N=f$ for every $f\in L^\infty(H)$.
  \item If $g\in \CB(K)$, then $g^N\in \CB(H)$.
\end{enumerate}
\end{lemma}
\begin{proof}
 Let $\lambda_N$ be the normalized Haar measure on $N$.
For $g\in L^\infty(K)$, we define a function $N^g$ on $K$ by \[N^g(x)=\int_N g(xn)d\lambda_N(n)\quad\text{for }\quad x\in K.\]

It is easy to see that $N^g\in L^\infty(H).$  Due to the invariance of $\lambda_N$ on $N$,  $N^g$ is constant on each coset of $N.$ It follows that $N^g$  defines a function on $H$,  we define a function $g^N$ on $H$ by setting
$g^N(\varphi(x))=N^g(x)$. Clearly, $\|g^N\|\leq \|g\|$.

Let now $f\in L^\infty(H)$, we check then that  $(f\circ \varphi)^N=f.$ Let $y\in H$ and pick $x\in K$ with $\varphi(x)=y.$
Then \begin{align*}(f\circ\varphi)^N(y)&=N^{(f\circ\varphi)}(x)=\int_N f\left(\varphi(xn)\right)d\lambda(n)\\&
=\int_N f(\varphi(x)\varphi(n))d\lambda(n)=\int_N f(\varphi(x))d\lambda(n)=f(\varphi(x))=f(y),\end{align*}
as required.

Let finally $g\in \CB(K)$. We first check that  $N^g\in \CB(K)$. To see this, it is enough to check that the function is continuous on every relatively compact open set.
So let $U$ be such a set, and for each $x\in \overline U$, let $F(x)\in \CB(N)$
be given by $F(x)(n)=g(xn)$. Then by \cite[Lemma A.9]{BJM}, the function \[x\mapsto \mu(F(x)): \overline U\to \mathbb C\]
 is continuous
for each $\mu\in \CB(N)^*.$ In particular, the function  $N^g$ is continuous on  $U$, as required.

Since $\varphi$ is continuous and open,  $g^N$ must also be continuous. To see this let  $\epsilon>0$, $y\in H$ and $x\in K$ such that $\varphi(x)=y$ be arbitrarily chosen.
 Take  a neighbourhood $U$  of $x$ in $K$ such that $|N^g(u)-N^g(x)|<\epsilon$
for every $u\in U.$ Since $\varphi$ is open, we may take $V=\varphi(U)$ as a neighbourhood of $y$
in $H$ and  $|g^N(z)-g^N(y)|<\epsilon$ for every $z\in V.$ Thus, $g^N\in\CB(H)$.\end{proof}


\begin{lemma} \label{chou}
Let $H$ and $K$ be locally compact groups and suppose that
there exists a continuous, open and  surjective, homomorphism $\varphi:K\to H$ with compact kernel. Let in addition $\alpha$ be a cardinal number. The natural maps induced by $\varphi$:
\begin{align*}
  \Phi&:L^\infty(H)\to L^\infty(K),\\
\Phi_6&:\ell^\infty(\alpha,L^\infty(H)\to\ell^\infty(\alpha,L^\infty(K)) \mbox{ and }\\ \mathbf\Psi_6&:\frac{\ell^\infty(\alpha,L^\infty(H))}{\ell^\infty(\alpha,\CB(H))}\to\frac{\ell^\infty(\alpha, L^\infty(K))}{\ell^\infty(\alpha,\CB(K))}
\end{align*}
are all linear isometries.
 \end{lemma}

 \begin{proof} We notice  first
 that, putting  $N=\ker\varphi$, $H$ is topologically isomorphic to $G/N$.
  We may then apply \cite[Theorem 2.64]{Fo} to see that a subset $A\subset H$ is locally null if and only if $\varphi^{-1}(A)$ is locally null.
 It follows easily that the function $\Phi$, defined by $\Phi(f)=f\circ \varphi$ is a well-defined isometry.

We now apply Lemma \ref{lem:quotiso} to the linear isometry $\Phi$. Since $\varphi$ is continuous $\Phi(\CB(H))\subset \CB(K)$.  To check property \eqref{**} of Lemma \ref{lem:quotiso}, we choose for every  $g\in \CB(K)$, the function $g^N\in \CB(H)$ constructed in Lemma \ref{weil}. Applying Lemma \ref{weil},
 \begin{align*}\|\Phi(f)-g\|\ge \|(\Phi(f)-g)^N\|&=
\|(\Phi(f))^N-g^N\|\\&=\|f-g^N\|\quad\text{for every}\quad f\in L^\infty(H).\end{align*}
Two  applications of Lemma \ref{lem:quotiso}, the first to the subspaces $F_1=F_2=\{0\}$ and the second to    the subspaces $F_1=\CB(H)$ and $F_2=\CB(K)$, show that both $\Phi_6$ and
$\mathbf{\Psi_6}$ are linear isometries.
\end{proof}

The material which follows  may be known to the experts in measure theory, but we have found no precise reference in the literature. Since the material is needed for the construction of the fourth isometry,  it is collected here for completeness.

We start with \cite[page 407]{R}, and recall that two measure spaces $(X, \mathcal A,\mu)$
and $(Y, \mathcal B,\nu)$ are isomorphic when there there is a bijection $\varphi:X\to Y$ such that \begin{align*}\varphi(A)\in \mathcal B\;\text{ and}\; \nu(\varphi(A))&=\mu(A)\quad\text{for every}\quad A\in\mathcal A\\
\varphi^{-1}(B)\in \mathcal A\;\text{ and}\; \mu(\varphi^{-1}(B))&=\nu(B)\quad\text{for every}\quad B\in\mathcal B\end{align*}

\begin{theorem} \label{Royden}
\cite[Chapter 15, Theorem 16]{R}
Let $X$ be an uncountable, complete, separable metric space and $\mu$ be an atomless, positive, finite, Borel measure on $X.$
Then the measure space $(X,\mu)$ is isomorphic to the standard measure space $(\I, \lambda),$ where  $\I$ is the closed unit  interval and $\lambda$ is the  Lebesgue measure.
\end{theorem}

Let $(X_1,\mu_1),\ldots,(X_n,\mu_n)$ be  atomless, positive, finite, Borel measure spaces on uncountable complete separable metric spaces. Then  the isomorphisms \[\varphi_i:\I\to X_i \quad (i=1,\ldots, n)\] given by Theorem \ref{Royden}
induce the following linear isometry and yield immediately the corollary below.
\begin{align*}T\colon L^1\left(\prod_{i=1}^n X_i,\bigotimes_{i=1}^n \mu_i\right)&\to
   L^1\left(\I^n,\bigotimes_{i=1}^n\lambda\right)\\ f&\longmapsto T(f),\end{align*}
   where $T(f)$ is given by
\[T(f)(t_1,\ldots, t_n)= f(\varphi_1(t_1),\ldots, \varphi_n(t_n))\quad\text{for}\quad (t_1,\ldots, t_n)\in\I^n.\]

\begin{corollary}\label{finite}
  Let $(X_1,\mu_1),\ldots,(X_n,\mu_n)$ be  atomless, positive, finite Borel  measures on uncountable, complete,  separable  metric spaces. The map $T$ is a linear surjective isometry.
    \end{corollary}

  If $\{X_i\}_{i\in I} $ is a family of non-empty sets and $F
   =\{i_1,\ldots,i_n\}$
   is a finite subset of $I$,
   we say that a function $f:\prod_{i \in I}X_i\to \C$ {\it depends on $F$}
if $f(\mathbf{x})=f(\mathbf{y})$ whenever $\mathbf{x}$ and $\mathbf{y}$ coincide on the coordinates 
   in $F$.

  \medskip

  When $f:\prod_{i \in I}X_i\to \C$ depends on some finite subset $F$ of $I,$ we may  define $f_F$ on $(X_{i_1}\times\cdots \times X_{i_n})$ as \[f_F(x_1,\ldots,x_n)=f(x_1,\ldots,x_n,\mathbf{0}), \]
  where $(x_1,\ldots,x_n,\mathbf{0})$ stands for the element of $\prod_i X_i$ with $0$ in every coordinate distinct from $\{i_1,\ldots,i_n\}$ and $x_j$ in the $i_j$-coordinate.

When the family $\{(X_i,\mu_i)\}_{i\in I}$ is made of measure spaces with  $\mu(X_i)=1$, for every $i\in I$,
we consider the infinite product measure space $X=\prod_{i\in I}X_i$ with the infinite product measure
$\mu=\bigotimes_{i\in I}\mu_i$ (as constructed  for instance in Section 28  of \cite{Ha}). If each $X_i$ is a compact group and $\mu_i$ is  its normalized Haar measure, then  $\bigotimes_{i\in I}\mu_i$ is again the normalized Haar measure of  the compact group $X=\prod_{i\in I}X_i$.

 The following  discussion leading to Lemma \ref{prod} follows the path of \cite[Page 42]{Fo}. For each finite subset $F$ of $I,$ we define
$L^1_F(X,\mu)$ as the space of all functions $f$ in $L^1(X,\mu)$
which depend on $F$ and such that $f_F\in L^1\left(\prod_{i\in F}X_i,\otimes_{i\in F}\mu_i\right).$
To simplify the notation  we shall not write the measures explicitly.
Thus,
\[L^1_F(X)=\left\{f\in L^1(X)\colon \txt{ $f=g$ (a.e) for some $g$  that depends\\ on $ F$ and $g_F\in L^1(\prod_{i\in F}X_i)$}\right\}.\]

  With $\mathcal F$ denoting the family of all finite subsets of $I$, we also define
  \begin{align*} L^1_{\mathcal F}\left(X\right)&=\bigcup_{F\in\mathcal F}
L^1_F(X)  \\&=\left\{f\in L^1\left(X\right)\colon  \txt{$f$
is equal a.e. to a function that  \\  depends
    on finitely many coordinates in $I$}\right\}.\end{align*}
  Then we have the following lemma.
%
%
%
%
%
%

\begin{lemma}
  \label{prod}
  Let $F\in \mathcal F$ and $\pi_F^X \colon  X \to \prod_{i\in F} X_i$ be the projection.
  \begin{enumerate}
\item   If $f\in  L^1_{F}(X)$, then
  $ \|f\|_1=\|f_F\|_1.$
 \item If $h\in  L^1\left(\prod_{i\in F} X_i\right)$, then ${\displaystyle h\circ\pi_F^X \in L^1_F\left(X\right)}$ and
  $ \| h\circ \pi_F^X \|_1=\|h\|_1.$
  \end{enumerate}
\end{lemma}

  \begin{theorem} \label{Psi4} Let $\{(X_i,\mu_i)\}_{i\in I}$ be a family of measure spaces, where each $X_i$
  is an uncountable complete separable metric space and each $\mu_i$ is an atomless, positive, finite, Borel measure measure on $X_i.$
  Then there is a linear isometry
 \[ \Psi \colon L^1\left(\prod_{i\in I}X_i \right)\to  L^1\left(\I^I \right),\]
 where $\I$ is the unit interval equipped with Lebesgue measure $\lambda$.
\end{theorem}

\begin{proof}
The notations introduced in the discussion that precedes this Lemma will be used throughout this proof, so that $\varphi_i \colon \I \to X_i$ will denote the map given by Theorem \ref{Royden} and, for any $F\subset I$ with $|F|=n$,  $T\colon L^1\left(\prod_{i\in F}X_i\right)\to L^1(\I^n)$
will be the isometry introduced immediately afterwards.
To avoid confusion, we denote the projection $\I^I\to \I^{F}$ simply by $\pi_F$.

%

We first define a linear map
\[T_0 \colon L^1_{\mathcal F}\left(\prod_{i\in I}X_i \right)\to  L^1_{\mathcal F}\left(\I^I \right)\] by
\[T_0(f)\Bigl(\left(t_i\right)_{i\in I}\Bigr)= f\Bigl(\left(\varphi_i(t_i)\right)_{i\in I}\Bigr).\]

The properties of $T_0$ are better understood once one realizes that given $f\in L^1_{F}\left(\prod_{i\in I}X_i \right)$,

\[T_0(f)= T(f_F)\circ \pi_F,\]

%
%

With this definition, it is clear that $T_0$  is well-defined (if $f=0$ a.e., then $f_F=0$ a.e.).



%
%

By Lemma \ref{prod} and Corollary \ref{finite}, we have that
 \[ \|f\|_1=\|f_F\|_1=\|T(f_F)\|=\|T(f_F)\circ \pi_F\|,\]
and we see that the map $T_0$  is   an isometry.

Now if $h\in L^1_F\left(\I^I \right)$, then $h_F\in L^1\left(\I^F\right)$. So, by Corollary  \ref{finite}, there is
 $g\in L^1\left(\prod_{i\in F}X_i\right)$ such that $T_F(g)=h_F$.
 Put $f=g\circ \pi_F^X$. Then  $f\in L^1_F(X)$, and  an easy check shows that
  $T_0(f)=h$.
%
%

Therefore  $T_0$ is  a linear isometry  mapping from $L^1_{\mathcal F}\left(\prod_{i\in I}X_i \right)$ onto
 $L^1_{\mathcal F}\left(\prod_{i\in I}\I_i \right)$.
      Since $L^1_{\mathcal F}\left(X\right)$ and $L^1_{\mathcal F}\left(\I^I\right)$  are dense, respectively, in $L^1\left(X\right)$ and  $L^1\left(\I^I\right)$        (see Lemma \ref{dense} below), the claim follows.
\end{proof}

 \begin{lemma}\label{dense} $L^1_{\mathcal F}\left(X\right)$ is dense in $L^1(X)$.
  \end{lemma}

          \begin{proof} Let $\CK(X)$ be the space of continuous functions on $X$ having a compact support
          and let $\CK_{\mathcal F}(X)$ be the space of functions in $\CK(X)$ which depend on some finite subset $F$ of $I.$ Since $L^1_{\mathcal F}\left(X\right)$ contains  $\CK_{\mathcal F}(X)$ which is dense in $\CK(X)$ (Stone-Weierstra\ss) and the latter is dense in $L^1(X)$, the claim follows.
\end{proof}

We next  isolate a set of  conditions that go in the spirit of Corollaries  \ref{cor:mainvv} and \ref{prodenar} and suffice to guarantee a linear isometric copy of $L^\infty(G)$ in the quotient space ${\dsp\frac{L^\infty(G)}{\CB(G)}}$, and so the extreme non-Arens regularity of the group algebra.

The mappings $\varphi_1$ and $\varphi_2$ of Lemma \ref{lem:cond}, that  will be explicitly produced in Theorem \ref{compenar}, induce the rest of the seven isometries, namely
$\mathbf\Psi_6$ and $\mathbf\Psi_2$.

\begin{lemma}
  \label{lem:cond}
Let $K$ be a locally compact group and suppose that four locally compact $\sigma$-compact groups $H_1,\:H_2,\:M_1$ and $M_2$ can be found such that:
 \begin{enumerate}
     \item $M_1$ and $M_2$ are metrizable, $\sigma-$compact and non-discrete,
     \item There is a surjective linear isometry $\Psi\colon L^\infty(H_2)\longrightarrow L^\infty(H_1)$.
         \item There exist   two  continuous and  open surjective homomorphisms\[  \xymatrix{M_2\times H_2 \ar[r]^{ \txt{\;\;$\displaystyle \varphi_2$}}& K\ar[r]^{\txt{$\displaystyle \varphi_1\;\;$}} & M_1\times H_1}\]
         such that $\ker(\varphi_1)$ is compact and
$\varphi_2( \lambda_{M_2}\otimes
\lambda_{H_2})=\lambda_{K}$ and $\varphi_1(\lambda_{K})=\lambda_{M_1}\otimes \lambda_{H_1}$.
   \end{enumerate}
If $G$ is any locally compact group having  an open subgroup Haar homeomorphic to $K$, then
 there exists a linear isometric copy of $L^\infty(G)$ in the quotient space ${\displaystyle \frac{L^\infty(G)}{\CB(G)}}.$
 In particular, $L^1(G)$ is extremely non-Arens regular.
\end{lemma}
\begin{proof} Let  $|G:K|=\alpha$.
As in the proof of Corollary \ref{prodenar}, we have  the isometric isomorphisms proved in \eqref{isom0} and \eqref{isom}
\begin{align*}\mathbf\Psi_1\colon& \xymatrix{ \dsp L^\infty(G)\ar[r]&\dsp \ell^\infty\left(\alpha,L^\infty(K)\right)  \quad\text{and}}\\
\dsp \mathbf\Psi_7\colon&\dsp   \xymatrix{ \dsp \frac{\ell^\infty\left(\alpha,L^\infty(K)\right)}{
\ell^\infty\left(\alpha,\CB(K)\right)}\ar[r]&\dsp\frac{L^\infty(G)}{\CB(G)}}.\end{align*}

 By Lemma \ref{chou} the maps  $\varphi_1$  and $\varphi_2$
induce  the second and  sixth linear isometries:
 \begin{align*}\mathbf\Psi_2\colon& \xymatrix{ \ell^\infty\left( \alpha,L^\infty(K)\right) \ar[r]& \ell^\infty\left( \alpha,L^\infty(M_2\times H_2)\right) \mbox{ and }}\\
%
\mathbf\Psi_6\colon& \xymatrix{
\displaystyle \frac{\ell^\infty\left(\alpha,L^\infty(M_1\times H_1)\right)}{\ell^\infty\left(\alpha,\CB(M_1\times H_1)\right)}\ar[r]&\displaystyle \frac{\ell^\infty\left(\alpha,L^\infty(K)\right)}{\ell^\infty\left(\alpha,\CB(K)\right)}}.\end{align*}


Next, we recall the isometry $\mathbf\Psi_5$ given by  Corollary \ref{cor:mainvv},
\[\mathbf\Psi_5\colon \ell^\infty\left(\alpha,L^\infty(H_1)\right)\longrightarrow \frac{\ell^\infty\left(\alpha, L^\infty(M_1\times H_1)\right)}{\ell^\infty\left(\alpha, \CB(M_1\times H_1)\right)},\]
and the isometry $\mathbf\Psi_3$ given by Lemma \ref{Lintoell} (ii),
\[\mathbf\Psi_3:
\ell^\infty\left( \alpha,L^\infty(M_2\times H_2)\right)
\longrightarrow \ell^\infty\left( \alpha,L^\infty(H_2)\right).\]

Finally, it is clear that the surjective linear isometry $\Psi\colon L^\infty(H_2)\longrightarrow L^\infty(H_1)$ induces a
map \[\mathbf\Psi_4\colon \ell_\infty(\alpha,L^\infty(H_2))\longrightarrow \ell_\infty(\alpha,L^\infty(H_1))\] with the same properties.

Putting everything together, we obtain the following chain of isometries that imply our claim,
and so the extreme non-Arens regularity of $L^1(G)$.

\[\begin{CD}
\dsp \frac{L^\infty(G)}{\CB(G)}@<\mathbf\Psi_7<<
\dsp \frac{\ell^\infty\left(\alpha,L^\infty(K)\right)}{\ell^\infty\left(\alpha, \CB(K)\right)}
@<\mathbf\Psi_6<<
\frac{\dsp \ell^\infty\left(\alpha, L^\infty(M_1\times H_1)\right)}{\dsp \ell^\infty\left(\alpha,\CB(M_1\times H_1)\right)}\\ @. @.
 @A \mathbf\Phi_5 AA   \\ @. @.
  \ell^\infty\left(\alpha,L^\infty(H_1)\right)\\ @. @.
 @A \mathbf\Psi_4 AA   \\ @.
\ell^\infty\left(\alpha,L^\infty(M_2\times H_2)\right)
@>\mathbf\Psi_3>>
\ell^\infty\left(\alpha,L^\infty(H_2)\right)
\\@.
@A\mathbf\Psi_2 AA  \\@.
\dsp \ell^\infty\left(\alpha,L^\infty(K)\right)\\@.
@A\mathbf\Psi_1 AA  \\
@. L^\infty(G).
\end{CD}\]
\medskip

This completes the proof.
\end{proof}

\begin{corollary}\label{cor:met}
  If $G$ is a locally compact metrizable non-discrete group, then
 there exists a linear isometric copy of $L^\infty(G)$ in the quotient space ${\displaystyle \frac{L^\infty(G)}{\CB(G)}}.$
 In particular, $L^1(G)$ is extremely non-Arens regular.
\end{corollary}
\begin{proof}
Choose $U$ a compact neighbourhood  and let $K=\langle U\rangle$, the (open) subgroup generated by $U$. Putting $M_1=M_2=K$, $H_1=H_2=\{e\}$, $\phi_1=\phi_2=id$,  we see that conditions (i), (ii) and (iii) of Lemma \ref{lem:cond} are fulfilled. The conclusion of the aforementioned Lemma yields the present corollary.
\end{proof}

\section{The final step}
We are now ready to prove Theorems A and B.
Corollary \ref{prodenar} implies in particular that every compact group of the form $M\times H$, where $M$ is a non-discrete, metrizable, $\sigma$-compact group, satisfies Theorems A and B and Corollary \ref{cor:met}  removes  the $\sigma$-compactness condition.

In view of Lemma \ref{lem:cond}, the main obstacle towards a complete proof is
in compact non-metrizable groups. But
Grekas and Mercourakis prove in \cite{GM} that every compact group $K$ may be sandwiched, from the measure and topological point of view,   between
two groups of the form $M\times H$ with $M$ non-discrete and metrizable.
With Lemma \ref{lem:cond}, we are then able to prove in Corollary \ref{cor:compenar} that Theorems A and B hold for every compact group.
Then we apply structural results by Davis \cite{davis55}  and Yamabe \cite{Ya} to see that any locally compact group contains an open subgroup of the form $\mathbb R^n\times K$, where $K$ is a compact group.
Together with Corollary \ref{prodenar} and Corollary \ref{cor:compenar}, this will lead to the full Theorems A and B.

We start by recalling Grekas and Mercourakis results, which will help us to deal with non-metrizable compact groups.

\begin{lemma}
\cite[Theorems 1.1(1) and 1.4(1)]{GM} \label{gm1}
Let $K$ be a infinite, compact, connected  group. There are two families of compact connected  metrizable groups $\{M_i\colon i\in I\}$ and $\{N_i\colon i\in I\}$ and two continuous and open surjective homomorphisms $\varphi_1$ and $\varphi_2$
\[ \prod_i M_i   \xrightarrow{ \varphi_2} K \xrightarrow{\varphi_1}  \prod_i N_i\] such that
$\varphi_2(\otimes_i \lambda_{M_i})=\lambda_K$ and $\varphi_1(\lambda_K)=\otimes_i \lambda_{N_i}$.
\end{lemma}

\begin{lemma}
\cite[Theorem 1.1 (2)]{GM} \label{gm2}
Let $K$ be a  totally disconnected, compact group. There is a homeomorphism $\phi\colon K\to \prod_i F_i$, where each $F_i$ a  finite group, such that $\varphi(\lambda_K)=(\otimes_i \lambda_{F_i}).$
\end{lemma}

\begin{theorem}
  \label{compenar}
  Let $G$ be an infinite, non-discrete, locally compact group having a compact open subgroup $K$.
  Then there exists a linear isometric copy of $L^\infty(G)$ in the quotient space ${\dsp\frac{L^\infty(G)}{\CB(G)}}.$
    \end{theorem}

\begin{proof}

If $G$ is metrizable, Corollary \ref{cor:met}  yields the Theorem.

Suppose now that $G$ is not metrizable. Then $K$ is not metrizable either.
As a compact group, $K$ is Haar homeomorphic to $K_0\times (K/K_0)$, where $K_0$ denotes the connected component of $K$ (see \cite[Theorem (B) and the remarks thereafter]{GM}, see also \cite[Theorem 8]{M}).
Then,  either $K_0$ or $K/K_0$ is not metrizable.

\emph{We suppose first that $K/K_0$ is metrizable.} Then $K_0$ is non-metrizable and so it must be non-discrete.
By Lemma \ref{gm1} we have two infinite collections  of compact connected  metrizable groups \[\{M_i\colon i\in I\}\quad\text{ and}\quad\{N_i\colon i\in I\}\] and two continuous and open surjective homomorphisms $\phi_1$ and $\phi_2$
\[  \xymatrix{\prod_i M_i  \ar[r]^{\displaystyle \;\;\;\phi_2}& K_0\ar[r]^{\displaystyle\phi_1\;\;\;} & \prod_i N_i}\] such that
$\phi_2(\otimes_i \lambda_{M_i})=\lambda_{K_0}$ and $\phi_1(\lambda_{K_0})=\otimes_i \lambda_{N_i}$.

We check  that $G$ satisfies the conditions of  Lemma \ref{lem:cond}.
We start by fixing any $i\in I$ and taking $M_i$ and $N_i$. Then these groups are metrizable and non-discrete,
and so are $M_i\times (K/K_0)$ and $N_i\times (K/K_0)$. Then let $J=I\setminus\{i\}$ and put
\begin{align*} &M_1=N_i\times (K/K_0)\quad\text{and}\quad M_2=M_i\times (K/K_0),\\&
H_1=\prod_{i\in J} N_i\quad\text{and}\quad
H_2=\prod_{i\in J} M_i,\\
  &M_2\times H_2=\prod_i M_i  \times(K/K_0)\xrightarrow{\;\;\;\varphi_2\;\;\;} K\approx K_0\times  K/K_0
  \xrightarrow{\;\;\;\varphi_1\;\;\;}  \prod_i N_i\times (K/K_0)=M_1\times H_1,\end{align*}
where $\varphi_2=\phi_2 \times \mathrm{id}_{K/K_0}$  and  $\varphi_1=\phi_1 \times \mathrm{id}_{K/K_0}.$
Condition (i) of Lemma \ref{lem:cond} is trivially  fulfilled for $M_1$ and $M_2.$
Since $H_1$ and $H_2$ are products of compact connected metric spaces indexed by the same set $J$,  the spaces $L^\infty(H_1)$ and $L^\infty(H_2)$ are
linearly  isometric by Theorem \ref{Psi4} so that  condition (ii) holds as well.
The maps $\varphi_1$ and $\varphi_2$  clearly satisfy condition (iii) of Lemma \ref{lem:cond}.
We can therefore apply the conclusion of Lemma \ref{lem:cond} to deduce that there is a linear isometric copy of $L^\infty(G)$ in the quotient space ${\dsp\frac{L^\infty(G)}{\CB(G)}}$.

\emph{Assume now that $K/K_0$ is not metrizable.}
By Lemma \ref{gm2} we have now an uncountable collection of finite groups $\{F_i\colon i\in I\}$ and a Haar  homeomorphism $\varphi\colon K/K_0\longrightarrow \prod_{i\in I}F_i$.
Fix an infinite countable subset $J$ of $I$ and let \[M=\prod_{i\in J}F_i\quad\text{and}\quad H=K_0\times\prod_{i\in I\setminus J}F_i.\]
 Then $M$ is a compact, non-discrete and metrizable subgroup of $K$ and K is Haar homeomorphic to $M\times H$.
Corollary \ref{prodenar} yields the desired claims.
 \end{proof}

\begin{remark}
 When $\kappa(G)\ge w(G),$ the extreme non-Arens regularity of $G$ may also  be obtained  by  measuring the size of the quotient $\frac{\luc(G)}{\wap(G)}$ as already done in  \cite{BF}, or the size of the quotient $\frac{\CB(G)}{\luc(G)}$ if $G$ is not discrete  as done in \cite{FG1}.
In \cite{FG1}, the proof is based on the construction by induction of a family $\mathcal F$ of subsets $\{T_\eta:\eta<\kappa(G)\}$
of $G$ which are uniformly disjoint such that each member fails to be an $\luc$-interpolation set but the union is
an approximable $\CB(G)$-interpolation set. The same method was used in \cite{BF}, where each member of the family $\mathcal F$ fails to be a $\wap(G)$-interpolation set but their union is
an approximaable $\luc(G)$-interpolation set.
In each case, this gives a linear isometric copy
of $\ell^\infty(\kappa(G))$ in the quotient space. Since $\kappa(G)\ge w(G)$, this gives a linear isometric copy
of $L^\infty(G)$ in the quotient space.
Since the quotient ${\dsp\frac{L^\infty(G)}{\CB(G)}}$ is trivial when $G$ is discrete, we shall use in our final section \cite{BF} or \cite{FV}
for the extreme non-Arens regularity of the group algebra.

But as already noted, these methods of proof (finding a copy of $\ell_\infty(\gamma)$ in the quotient ${\dsp\frac{L^\infty(G)}{\CB(G)}}$ and embedding  $L^\infty(G)$  in another copy of $\ell_\infty(\gamma)$) are likely to fail in large compact groups, as  for $\gamma>\omega$, $\ell_\infty(\gamma)$ does not embed in $L_\infty(G)$ regardless of the size of $G$ (see \cite[Proposition 4.7 and Theorem 4.8]{rose70}).
\end{remark}

\begin{corollary}\label{cor:compenar}
  If $G$ is an infinite  compact group, then there exists a linear isometric copy of $L^\infty(G)$ in the quotient space ${\dsp\frac{L^\infty(G)}{\CB(G)}},$ and so $L^1(G)$ is extremely non-Arens regular.
\end{corollary}

The study of $L^1(G)$ for any infinite locally compact group $G$ can now be reduced to that of $L^1(\mathbb R^n\times K)$ with $K$ compact, by means of the following Lemma due to Davis \cite{davis55} that builds on a  basic structural result of Yamabe \cite{Ya}.

\begin{lemma}[Davis, \cite{davis55}]\label{lem:davis}
Every  locally compact group $G$ contains an open subgroup $H$ which is Haar homeomorphic to $\R^n \times K$ where  $K$
is
a  compact group.
  \end{lemma}
\begin{proof}
By  a theorem of Yamabe    \cite{Ya}, $G$ contains an open subgroup $H$ which,  on its turn, has   a  compact normal subgroup $N$  such that $H/N$ is a connected Lie group. Lemma 2 of \cite{davis55}  then proves that $H$ admits the required decomposition.
\end{proof}

We reach now our aim. We start with Theorem B.

\begin{theorem} \label{Eureka}
  Let $G$ be an infinite, non-discrete, locally compact group.
  Then there exists a linear isometric copy of $L^\infty(G)$ in the quotient space ${\dsp\frac{L^\infty(G)}{\CB(G)}}.$
   \end{theorem}

\begin{proof}
  By the  previous lemma, $G$ contains an open subgroup $H$  of the form $H=\R^n\times K$.

If $n\neq 0$,  Corollary \ref{prodenar} gives the required isometry.
If, on the contrary, $n=0$, then $H=K$ is a non-discrete, compact subgroup of $G$  and Theorem \ref{compenar} yields the same conclusion.
\end{proof}

And here is Theorem A.

\begin{theorem} \label{Eureka!}
  Let $G$ be an infinite locally compact group.
  Then $L^1(G)$ is extremely non-Arens regular.
 \end{theorem}

 \begin{proof} If $G$ is non-discrete, this is an immediate consequence of Theorem \ref{Eureka}.
 If $G$ is discrete, the claim was proved in \cite{FV} or \cite{BF}.
 \end{proof}

Let $A(G)$ be the Fourier algebra, consisting of all functions $u\in C_0(G)$ of the form $f\ast\check g$ $(x\in G)$, where $f,g\in L^2(G)$ and $\check g(x)=\overline{g(x^{-1})}.$
The norm in $A(G)$ is given by \[\|u\|=\inf\{\|f\|_2\|g\|_2:\;u=f*\check g, f, g\in L^2(G) \}\]
(see Eymard \cite{Eymard64}).
In harmonic analysis,  $A(G)$ is seen as the dual
object of the group algebra $L^1(G)$. In fact when $G$ is abelian, the Fourier transform is an isometric
algebra isomorphism from $L^1(\widehat G)$ onto $A(G).$

As already noted at the begining of the paper, the extreme non-Arens regularity
of the Fourier algebra $A(G)$ was proved by Hu in \cite{H} when $w(G)\ge \kappa(G).$
The question of whether $A(G)$ is extremely non-Arens regular for any infinite locally compact group remains open.
But when $G$ is Abelian, we can omit the condition $w(G)\ge \kappa(G).$

\begin{corollary} If $G$ is an infinite locally compact Abelian group, then the Fourier algebra $A(G)$
 is extremely non-Arens regular
 \end{corollary}

\subsection*{Acknowledgments}
Parts of the article was written when the first named author was visiting Jaume University in Castellon in December 2011
and May 2012. He would like to express his warm thanks for the kind hospitality and support.

\samepage

\end{document}